\documentclass[a4paper,11pt]{article}
\usepackage{amssymb,amsmath,latexsym,amsthm}
\usepackage{color,mathrsfs}
\usepackage{url}
\usepackage{enumerate}
\usepackage[square,sort,comma,numbers]{natbib}
\usepackage{dsfont}
\usepackage{diagbox}
 \usepackage{geometry}
\usepackage{graphicx}
\usepackage{colortbl}
\usepackage{bigints}
\usepackage{endnotes}

\renewcommand{\textbf}[1]{\begingroup\bfseries\mathversion{bold}#1\endgroup}

\usepackage{fancyhdr}
\usepackage{pstricks,pst-plot,pst-node,pstricks-add}

\newtheorem{thm}{Theorem}[section]
\newtheorem{defi}{Definition}[section]

\newtheorem{prop}[thm]{Proposition}
\newtheorem{Conjecture}[thm]{Conjecture}

\theoremstyle{definition}
\newtheorem{remark}[thm]{Remark}

\newcommand{\R}{\mathbb R}

\newcommand{\Z}{\mathbb Z}
\newcommand{\N}{\mathbb N}

\numberwithin{equation}{section}

\def\XXint#1#2#3{{\setbox0=\hbox{$#1{#2#3}{\int}$}
    \vcenter{\hbox{$#2#3$}}\kern-.5\wd0}}

\allowdisplaybreaks
\date{date}
\begin{document}
\title{Theta functions and optimal lattices for a grid cells model}
\author{Laurent B\'{e}termin\\ \\
Faculty of Mathematics, University of Vienna,\\ Oskar-Morgenstern-Platz 1, 1090 Vienna, Austria\\ \texttt{laurent.betermin@univie.ac.at}. ORCID id: 0000-0003-4070-3344 }
\date\today
\maketitle

\begin{abstract}
Certain types of neurons, called ``grid cells", have been shown to fire on a triangular grid when an animal is navigating on a two-dimensional environment, whereas recent studies suggest that the face-centred-cubic (FCC) lattice is the good candidate for the same phenomenon in three dimensions. The goal of this paper is to give new evidences of these phenomena by considering a infinite set of independent neurons (a module) with Poisson statistics and periodic spread out Gaussian tuning curves. This question of the existence of an optimal grid is transformed into a maximization problem among all possible unit density lattices for a Fisher Information which measures the accuracy of grid-cells
representations in $\R^d$. This Fisher Information has translated lattice theta functions as building blocks. We first derive asymptotic and numerical results showing the (non-)maximality of the triangular lattice with respect to the Gaussian parameter and the size of the firing field. In a particular case where the size of the firing fields and the lattice spacing match with experiments, we have numerically checked that it is possible to find a value for the Gaussian parameter above which the triangular lattice is always optimal. In the case of a radially symmetric distribution of firing locations, we also characterize all the lattices that are critical points for the Fisher Information at fixed scales belonging to an open interval (we call these lattices ``volume stationary"). It allows us to compare the Fisher Information of a finite number of lattices in dimension 2 and 3 and to give another evidences of the optimality of the triangular and FCC lattices.
\end{abstract}


\tableofcontents

\section{Introduction and setting}
\subsection{Presentation of the problem}

A challenging mathematical problem is to justify rigorously why periodic patterns arise in nature and experiments: densest packing \cite{Viazovska,CKMRV}, atoms in a solid \cite{Kantorovich,RadinLowT}, triangular lattices of Ginzburg-Landau vortices in type II superconductors \cite{Abrikosov,Sandier_Serfaty}, rich polymorphic behavior for systems with Coulombian interactions \cite{TrizacWigner16,LuoChenWei}, etc. This type of problem, also called ``Crystallization problem" (see \cite{BlancLewin-2015}), is often considered as variational, which means that optimal periodic structures can be seen as extrema of certain functionals usually derived from simplified models.

\medskip

A typical example in Neurobiology is the existence of grid cells in the medial entorhinal cortex (MEC) of the brain discovered by Hafting et al. \cite{Haftingetal} and that have been brought to light for many mammals (rats, humans, etc.), see e.g. \cite{Gridcellsreview}. Each of these neurons is tuned to the position of the animal and fires when it crosses the sites of an approximate triangular lattice (also called ``hexagonal lattice", see \eqref{eq:tri} for a precise formula and Figure \ref{fig:tri} for a representation) during  a two-dimensional navigation. Furthermore, several evidences have been found in the brains of flying bats and humans concerning face-centred-cubic (FCC) lattices encoding the spatial representation in three-dimensional displacements (see e.g. \cite{Bats3d,3DHumangrids}). It also appears that grid cells can be split into modules where the firing periodic patterns share the same scale and orientation but are shifted around an average position belonging to a firing field (i.e. the corresponding spatial region that evokes firing). Also, it is important to notice that this phenomenon occurs at different scales: the lattice spacings of the grids variy in a discrete way throughout the MEC as shown in \cite{MECdiscrete}. In this paper, we aim to design a model supporting the fact that one module is tuned to a triangular or a FCC pattern.

\medskip

Several Mathematical attempts to justify the emergence of the triangular grid for one module have been made (see e.g. \cite{GridPlaceCells,Towseetal,Gridcellseconomy,Mathisetal,dAlbis,Sorscheretal,Modelingfridcells,Anselmietal}), and the present work is inspired by the one of Mathis et al. \cite{Mathisetal} where an Information Theory point of view has been chosen. The goal is to maximize the Fisher Information's trace measuring the accuracy of grid cells representation (also called ``resolution"). The main novelty in our work compared to \cite{Mathisetal} is that our model does not simplify into a best packing problem leading to the triangular and FCC lattices as optimizers. Furthermore, we are considering more spread out Gaussian tuning curves instead of short-range bump functions. All these new assumptions transform the short-range packing problem derived in \cite{Mathisetal} into an energy maximization problem pretty much the same way as the hard-sphere crystallization question \cite{Rad2} and the best packing problems \cite{FejesToth,Viazovska,CKMRV} were transformed into a minimization question among lattices for Gaussian interactions in \cite{Mont,CKMRV2Theta}.

\medskip

We now briefly introduce our setting. The reader can refer to Section \ref{sec:Fisher} for a complete derivation and explanation of our main formula based on several assumptions that are only partially recalled in this introduction. The main statistical notions are also well explained in \cite{CoverThpas}. We assume that neurons (grid cells) fire independently following a Poisson distribution. The average firing number of a neuron is given by a lattice periodic tuning curve $\Omega^\alpha_L$ which tunes the position $x\in \R^d$ of the animal to the positive real number $\Omega^\alpha_L(x)\in \R_+$, where 
\begin{equation}
\Omega_L^\alpha(x)=\theta_{L+x}(\alpha):=\sum_{p\in L} e^{-\pi \alpha |p+x|^2},\quad \alpha>0,\quad \quad L=\bigoplus_{i=1}^d \Z u_i,
\end{equation}
the Euclidean norm being denoted by $|.|$ and $\{u_i\}_{1\leq i\leq d}$ being a basis of $\R^d$. Such infinite discrete set $L\subset \R^d$ is called a $d$-dimensional lattice and $\theta_{L+x}(\alpha)$ is called the translated lattice theta function (see also \cite{BeterminPetrache}) which quantifies a Gaussian interaction between a point $x\in \R^d$ and a lattice $L$ with Gaussian parameter $\alpha$. Notice that the constant $\pi$ is only here for technical reasons and to fit perfectly with the usual definition of theta functions. We call $\mathcal{L}_d(V)$ the set of lattices with co-volume $V$, i.e.
$$
\mathcal{L}_d(V)=\left\{L=\bigoplus_{i=1}^d \Z u_i : \{u_i\}_i \textnormal{ basis of $\R^d$}, |\det(u_1,...,u_d)|=V\right\}.
$$
Notice that the choice of a tuning curve of Gaussian type describes well (based on the central limit theorem) the spatial positions at which neurons fire, at the population level (see e.g. \cite{dAlbis,Gerlei}). Furthermore, since the translated theta function is also traditionally seen as a building blocks for a large class of lattice energies (see e.g. \cite{CohnKumar,BetTheta15,OptinonCM}), we have chosen to focus on this important Gaussian case in this paper. However, our main formula will be written in Section \ref{sec:Fisher} in terms of a general function $f$ instead of the Gaussian function $r\mapsto e^{-\pi \alpha r^2}$.

\medskip

Since a module is given by a set of translated copies of the lattice $\{L+y_j\}_{j}$ and tuning curves $\Omega^\alpha_{L+y_j}(x)$, we are assuming that the empirical measure $\delta_Y$ associated with the set of shifts $Y=\{y_j\}_j$ converges to a Radon measure $\mu$ (i.e. the firing distribution) with support $\Sigma$ (called ``firing field" in this context). We will write $\mu\in \mathcal{M}(\Sigma)$. It is well-known that $\Sigma$ scales as the lattices in the experiments (see e.g. \cite{ProgressiveScale}), in the sense that multiplying the lattice distances by $\lambda>0$ implies that we have to replace the firing field by $\lambda \Sigma$. In order to take this scaling effect into consideration, we define, for any $\lambda>0$,  the rescaled measure $\mu_\lambda$ on the Borelian set $\mathcal{B}_d$ of $\R^d$ by
\begin{equation}
\mu_\lambda(F)=\mu(\lambda F),\quad \forall F\in \mathcal{B}_d.
\end{equation}
 Figure \ref{fig:ex} depicts the situation in $\R^2$. Our goal is to decode the position (e.g. $x=0$) of the animal given the number of spikes of a neuron population. Therefore, as in \cite{Mathisetal}, we are considering the Fisher information $J_M(0)$ per neuron associated to the module $M:=(L,\Omega_L^\alpha,\mu,\Sigma)$ whose inverse of the trace is a lower bound for the error made in the decoding process (this is the so-called Cramer-Rao bound, see e.g. \cite[Thm. 11.10.1]{CoverThpas}). In our setting, the trace of $J_M(0)$ is shown (see Section \ref{sec:Fisher}) to be equal to (see also \eqref{eq:formulaFdetailed} for an expanded formula)
\begin{equation}\label{eq:Fmualpha}
\mathcal{F}^\alpha_{\mu}(L):=\int_\Sigma  \mathcal{Q}_L^\alpha(y)d\mu(y),\quad \textnormal{where}\quad    \mathcal{Q}_L^\alpha(y):=\left|\nabla_y \sqrt{\theta_{L+y}(\alpha)}\right|^2,
\end{equation}
and the question we are investigating is to find the lattice for which the decoding error is the smallest possible., i.e. the lattice $L$ maximizing $\mathcal{F}_\mu^\alpha$.

\begin{figure}[!h]
\centering
\includegraphics[width=8cm]{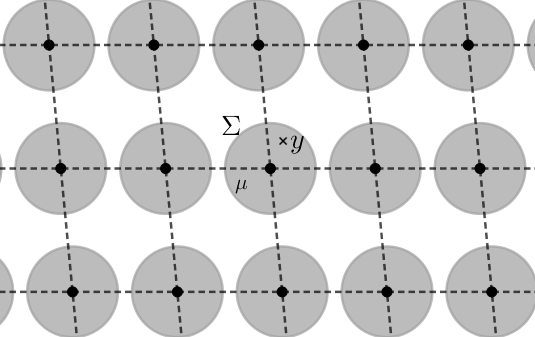} 
\caption{General configuration with arbitrary lattice $L$, firing field $\Sigma=B_R$ (repeated periodically) and a generic point $y\in \Sigma$.}
\label{fig:ex}
\end{figure}

\begin{figure}[!h]
\centering
 \includegraphics[width=8cm]{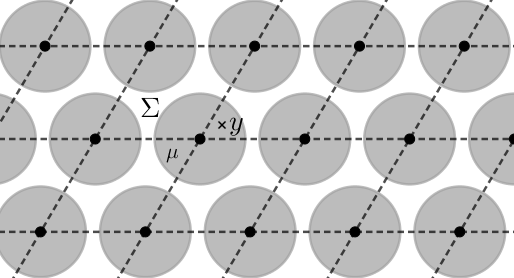}
\caption{Triangular configuration with lattice $\mathsf{A}_2$ defined by \eqref{eq:tri}, firing field $\Sigma=B_R$ (repeated periodically) and a generic point $y\in \Sigma$.}
\label{fig:tri}
\end{figure}

According to the following scaling formula (see Proposition \ref{prop:scale}) which hols for any $\alpha>0$, $\lambda>0$, $\Sigma\subset \R^d$, $\mu\in \mathcal{M}(\Sigma)$ and $L\in \mathcal{L}_d(1)$,
\begin{equation}\label{eq:scalingintro}
\mathcal{F}_{\mu_\lambda}^\alpha(\lambda L)=\lambda^{-2} \mathcal{F}_\mu^{\lambda^2 \alpha}(L),
\end{equation}
it is therefore enough to consider unit density lattices in our problem stated as follows.

\medskip

\noindent \textbf{Maximization Problem:} Given $\alpha>0$ and $\mu$, what is the maximizer of $\mathcal{F}^\alpha_{\mu}$ in $\mathcal{L}_d(1)$?

\begin{remark}[Size of firing fields and lattice spacing]\label{rmk:Brun}
It is implicit that $\Sigma$, i.e. the support of $\mu$, is also fixed in the above problem. Furthermore, it has been observed that firing fields of a single grid cell do not seem to overlap (see e.g. \cite[Fig. 1]{Haftingetal}). In particular, the ratio between the diameter of $\Sigma$ and the lattice spacing of the grid appears to be approximately constant along the MEC -- equal to $\approx 0.3$ for Brun et al. \cite{ProgressiveScale} -- and very recent experimental results by Nagele et al. \cite{Nageleetal} also suggest that this ratio might be actually slightly smaller (the firing fields are narrower). Our numerics in Section \ref{numerics} (see Figure \ref{fig:a10Rup}) will support  the fact that too much overlapping implies the non-existence of a maximizer for $\mathcal{F}_\mu^\alpha$ in $\mathcal{L}_2(1)$.
\end{remark}

In this paper, we give several types of results concerning this maximization problem: asymptotic, numerical and constrained. Indeed, it appears that finding a rigorous proof of any optimality result for $\mathcal{F}_\mu^\alpha$, where $\alpha,\mu$ are fixed, is out of reach (see for instance Section \ref{subsec:disc}). However, it is possible to study the behaviour of the functional $\mathcal{F}_\mu^\alpha$ as $\alpha\to 0$, or numerically when the parameters are fixed. Furthermore, staying in the class of lattices that are critical points of $\mathcal{F}_\mu^\alpha$ in an open interval of scales leads to comparing only a finite number of highly symmetric lattices.

\subsection{Some known results on optimal lattices}

The type of optimality problem investigated in this paper has recently attracted a lot of attention (see e.g. \cite[Sect. 2.5]{BlancLewin-2015}), especially in Mathematical Physics and for energies per point of type
\begin{equation}\label{eq:Ef}
E_f[L]:=\sum_{p\in L} f(|p|^2),\quad \textnormal{where}\quad |f(r)|=O(r^{-\frac{d}{2}-\eta}), \eta>0\quad \textnormal{as $r\to +\infty$},
\end{equation}
and where $f$ is interpreted as an interaction potential between particles. The exponential case -- i.e. where $f(r^2)=e^{-\pi \alpha r^2}$, $\alpha>0$, is a Gaussian function -- appears to be of fundamental importance. In this case, the energy defined in \eqref{eq:Ef} is called the lattice theta function 
\begin{equation}
\theta_L(\alpha):=\sum_{p\in L} e^{-\pi \alpha |p|^2},
\end{equation}
which appears to be a building block for any energy of type $E_f$ where $f=\mathcal{L}[\mu_f]$ is the Laplace transform of a measure $\mu_f$ on $\R_+$ (see e.g. \cite{CohnKumar,BetTheta15,OptinonCM}). When $\mu_f$ is a nonnegative measure, then $f$ is called completely monotone, which means that for all $r>0$ and all $k\in \N\cup \{0\}$, $(-1)^n f^{(n)}(r)\geq 0$ (this is a consequence of Hausdorff-Bernstein-Widder Theorem \cite{Bernstein-1929}). Therefore, the optimality of a lattice in $\mathcal{L}_d(1)$ for $L\mapsto \theta_L(\alpha)$ and for all $\alpha>0$ ensures the same result for $E_f$ when $f$ is completely monotone. If we allow our structures $L$ to be more general periodic configurations, then such optimality result is called universal optimality and has been recently solved in dimension $d\in \{8,24\}$ in \cite{CKMRV2Theta} where the best (densest) packings of spheres (see \cite{Viazovska,CKMRV}) $\mathsf{E}_8$ and $\Lambda_{24}$ are the unique minimizers in the respective dimensions. However, this problem is still open in dimension $d=2$ where the triangular lattice
\begin{equation}\label{eq:tri}
\mathsf{A}_2:=\sqrt{\frac{2}{\sqrt{3}}}\left[\Z(1,0)\oplus \Z\left(\frac{1}{2},\frac{\sqrt{3}}{2}  \right)  \right],
\end{equation}
which is the two-dimensional best packing (see \cite{FejesToth}), is conjectured to be the unique minimizer. When restricted to the set of lattices, the lattice theta function has been shown by Montgomery \cite{Mont} to be minimized in $\mathcal{L}_2(1)$ by $\mathsf{A}_2$ for any $\alpha>0$. Furthermore, such result cannot be true in dimension $d=3$ as explained in \cite{SarStromb} since the best candidates, which are the Face-Centred-Cubic (FCC) and Body-Centred-Cubic (BCC) lattices respectively defined by
\begin{align*}
&\mathsf{D}_3:= 2^{-\frac{1}{3}}\left[\Z(1,0,1)\oplus \Z(0,1,1)\oplus \Z(1,1,0)  \right]\\
& \mathsf{D}_3^*:=2^{\frac{1}{3}}\left[\Z(1,0,0)\oplus \Z(0,1,0)\oplus \Z\left(\frac{1}{2},\frac{1}{2},\frac{1}{2}  \right)  \right]
\end{align*}
are dual of each other  but not unimodular, i.e. $\mathsf{D}_3^*\neq \mathsf{D}_3$. We recall that the dual lattice of $L$ is defined by $L^*=\{y\in \R^d : y\cdot p\in \Z, \forall p\in L\}$. Let us also recall the results of \cite{BeterminPetrache,Beterminlocal3d} where $\mathsf{D}_3$ (resp. $\mathsf{D}_3^*$) has been shown to be a strict local minimizer of $L\mapsto \theta_L(\alpha)$ in $\mathcal{L}_3(1)$ for sufficiently large (resp. small) $\alpha$ and a saddle point for sufficiently small (resp. large) $\alpha$.

\medskip

The extrema of the translated theta function $(L,y)\mapsto\theta_{L+y}(\alpha)$ have also been studied. On the one hand, all the critical points of $y\mapsto \theta_{L+y}(\alpha)$ as well as their nature are known when $L$ is an orthorhombic (see \cite[Sect. III.3]{BeterminPetrache}) or triangular lattice (see \cite{Baernstein-1997}) and for all $\alpha>0$, see also Section \ref{subsec:disc}. For general lattices, only the trivial critical points are known (center of the unit cell and midpoints). Therefore, all the zeros of $\mathcal{Q}_L^\alpha$ are also known in these cases. On the other hand, only few results are available concerning the extrema of $L\mapsto \theta_{L+y}(\alpha)$ for fixed $y\neq 0$. In the case where $y=c_L$ is the center of the unit cell of $L$, $L\mapsto \theta_{L+c_L}(\alpha)$ does not have any minimizer (see \cite[Prop. 1.3]{BeterminPetrache}) and the maximizer among two-dimensional lattices (resp. among $d$-dimensional orthorhombic lattices) with the same density is the triangular lattice $\mathsf{A}_2$ (resp. the simple cubic lattice $\Z^d$) as we have shown in \cite{MaxTheta1} (resp. in \cite[Thm. 1.4]{BeterminPetrache} by mainly using tools from \cite{Faulhuber:2016aa}). Notice that other two-dimensional results of this type might be derived from new discoveries concerning combinations of theta functions by Luo and Wei in \cite{LuoWeiBEC}. 

\medskip

Finally, concerning the averaging of the shifts for the lattice theta functions, i.e. for energies of the form $\int_\Sigma \theta_{L+y}(\alpha)d\mu(y)$, recent advances have been made in \cite{BetKnupfdiffuse,BetSoftTheta} when $\mu$ is a radially symmetric measure sufficiently rescaled around the origin or generated by a completely monotone kernel $\rho$ as $d\mu(x)=\rho(|x|^2)dx$. In particular, the above mentioned results \cite{Mont,BeterminPetrache,Beterminlocal3d,CKMRV2Theta} for the lattice theta function $L\mapsto \theta_L(\alpha)$ on $\mathcal{L}_d(1)$ still hold in these particular cases for $L\mapsto \int_\Sigma \theta_{L+y}(\alpha)d\mu(y)$. The asymptotic results presented in the next section are actually based on these facts combined with the work of Regev and Stephens-Davidowitz \cite{Regev:2015kq} on the growth of (translated) lattice theta functions.

\subsection{Numerical and asymptotic results}

Whereas $\mathcal{F}_\mu^\alpha$ defined by \eqref{eq:Fmualpha} has the translated lattice theta function $\theta_{L+y}(\alpha)$ as a building block, it cannot a priori be considered like a lattice energy of type $E_f$ defined by \eqref{eq:Ef}. It is also clear that finding the global maximizer of  $\mathcal{F}_\mu^\alpha$ using an analytic or variational method appears to be out of reach but numerically accessible. Moreover, some asymptotic results can be shown.

\medskip

It is indeed usual to study the asymptotic behaviour with respect to $\alpha$ of a problem involving lattice theta functions (see e.g. \cite[Thm. 1.6]{BeterminPetrache}). When $\alpha\to 0$, an alternative formula (see \eqref{eq:altF}) using a recent result by Regev and Stephens-Davidowitz in \cite{Regev:2015kq} combined with new results on the soft lattice theta function \cite{BetKnupfdiffuse,BetSoftTheta} lead to the following asymptotic result.

\begin{thm}[\textbf{The small $\alpha$ case - Non-maximality results}]\label{thm:alpha0epsilon0}
 Let $\mu\in \mathcal{M}(\Sigma)$ be radially symmetric with support $\Sigma=B_R\subset \R^2$, for some $R>0$. Then:
\begin{enumerate}
\item For all $L\in \mathcal{L}_2(1)$, $\exists \alpha_0>0$, $\exists \lambda_0>0$, such that for all $\alpha\in (0,\alpha_0)$, for all $\lambda\in (0,\lambda_0)$, 
$$
\mathcal{F}^\alpha_{\mu_\lambda}(L)>\mathcal{F}^\alpha_{\mu_\lambda}(\mathsf{A}_2).
$$
\item For $\mu$ such that $d\mu(x)=\rho(|x|^2)dx$, where $\rho$ is a completely monotone function and all  $L\in \mathcal{L}_2(1)$  there exists $\alpha_0$ such that for all $\alpha\in (0,\alpha_0)$, $\mathcal{F}^\alpha_{\mu}(L)>\mathcal{F}^\alpha_{\mu}(\mathsf{A}_2)$.
\item There exists $\alpha_1>0$ and $\lambda_1>0$ such that for all $\alpha\in(0,\alpha_1)$ and $\lambda\in (0,\lambda_1)$, $\mathcal{F}^\alpha_{\mu_\lambda}$ does not have any maximizer in $\mathcal{L}_2(1)$.
\end{enumerate} 
 In dimension $d=3$, the two first points hold for $L$ in an open ball of $\mathcal{L}_3(1)$ centred at $\mathsf{D}_3^*$ and the third point also holds.
\end{thm}

This result agrees with the numerics we have performed for small $\alpha$ and where the triangular lattice minimizes $\mathcal{F}_\mu^\alpha$ (see Figure \ref{Gaussiana2R01}). Furthermore, our main numerical findings are the following (see Section \ref{numerics}), choosing $\Sigma=B_R$, $\mu=\sigma_R$ being the uniform measure on $B_R$ and $\alpha=\frac{10}{\pi}$:
\begin{itemize}
\item \textbf{Dimension 2. Maximality of the triangular lattice.}
\begin{itemize}
\item \textbf{Global optimality of the triangular lattice.} The triangular lattice $\mathsf{A}_2$ is the unique maximizer of $\mathcal{F}_\mu^\alpha$ in $\mathcal{L}_2(1)$ for $R\in \left\{0.1,0.2,0.3,0.4,0.5, \frac{1}{2}\sqrt{\frac{2}{\sqrt{3}}}\right\}$. The last value of $R$ corresponds to half of the side length of $\mathsf{A}_2$.
\item \textbf{Non-existence of a maximizer for large $R$.} If $R\geq 0.59$,  then $\mathcal{F}_\mu^\alpha$ does not have a maximizer in $\mathcal{L}_2(1)$ and $\mathsf{A}_2$ is its unique minimizer in $\mathcal{L}_2(1)$.
\item \textbf{Comparison with experimental values from \cite{ProgressiveScale}.} If $R=0.16$, which corresponds to the case where the ratio between the diameter of $\Sigma$ and the lattice spacing of the triangular lattice is $0.3$ (as suggested by \cite{ProgressiveScale}), we observe that $\mathsf{A}_2$ is the unique maximizer of $\mathcal{F}_\mu^{\bar{\alpha}}$ in $\mathcal{L}_2(1)$ if $\bar{\alpha}>\bar{\alpha}_0\approx 1.25$. For $\bar{\alpha}\leq \bar{\alpha}_0$, $\mathcal{F}_\mu^{\bar{\alpha}}$ does not seem to have a maximizer in $\mathcal{L}_2(1)$ and $\mathsf{A}_2$ is a minimizer in $\mathcal{L}_2(1)$.
\end{itemize} 
\item \textbf{Dimension 3. Maximality of the FCC lattice} 
\begin{itemize}
\item \textbf{Comparison of cubic lattices.} There exists $R_0\approx 0.57$ such that for all $R\leq R_0$, 
$$
\mathcal{F}_\mu^\alpha(\mathsf{D}_3)>\mathcal{F}_\mu^\alpha(\mathsf{D}_3^*)>\mathcal{F}_\mu^\alpha(\Z^3).
$$
\item \textbf{Local maximality of the FCC lattice.} $\mathsf{D}_3$ is a local maximizer of $\mathcal{F}_\mu^\alpha$ in $\mathcal{L}_3(1)$ for $R\in \{0.1,0.2,0.3,0.4,0.5, 2^{-5/6}\}$.  The last value of $R$ corresponds to half of the side length of $\mathsf{D}_3$.
\end{itemize}
\end{itemize}

According to the scaling formula \eqref{eq:scalingintro}, the same qualitative behaviour holds for any other value of $\alpha$ with rescaled quantities.

\subsection{Characterization of volume-stationary lattices}

In \cite{MECdiscrete}, it has been observed that rats have grid cells with grid spacing from 35.2cm to 171.7cm, from dorsal to ventral MEC. The set of possible scales has been shown to be discrete for all the studied species, with different values. If we believe that grid cells are universal among a lot of animals, in such a way that grid scales are covering all the possible values of a certain (open) interval, we can focus on lattices $L$ that are critical points (or maximizer) of  $\mathcal{F}_{\mu}$ in $\mathcal{L}_d(V)$ for all $V$ in an open interval. We show a result which is similar to the one obtained in \cite[Sect. III]{LBMorse} for energies of type $E_f$ defined by \eqref{eq:Ef}, using tools from \cite{DeloneRysh,Gruber}.

\medskip

We first recall the notion of strongly eutactic layer exactly as we have done it in \cite[Sect. III]{LBMorse}.

\begin{defi}[\textbf{Strongly eutactic layer}]\label{def:eutactic}
Let $L\in \mathcal{L}_d(1)$. We say that a layer $\mathfrak{m}=\{p\in L ; |p|=\lambda\}$, for some $\lambda>0$, of $L$ is strongly eutactic if $\sharp \mathfrak{m}=2k$ for some $k\in \N$ and, for any $x\in \R^d$, 
$$
\sum_{p\in \mathfrak{m}} \frac{(p\cdot x)^2}{|p|^2}=\frac{2k}{d}|x|^2.
$$
\end{defi}
\begin{remark}
After a suitable renormalization, $\mathfrak{m}$ is also called a spherical $2$-design (see e.g. \cite{BachocVenkov}).
\end{remark}

\begin{thm}[\textbf{Volume-stationary lattices for $\mathcal{F}^\alpha_{\mu}$}]\label{thm:density}
We assume that $\Sigma\subset \R^d$ is compact and $\mu\in \mathcal{M}(\Sigma)$ is radially symmetric. Then, $L\subset \R^d$ is a critical point of $L\mapsto \mathcal{F}_{\mu_\lambda}(\lambda L)$ in $\mathcal{L}_d(1)$ for all $\lambda$ in an open interval $I$ if and only if all the layers of $L$ are strongly eutactic and $I=(0,\infty)$. In particular, $L\in \{\mathsf{A}_2,\Z^2,\Z^3,\mathsf{D}_3,\mathsf{D}_3^*\}$ in dimension $d\in \{2,3\}$.
\end{thm}

Therefore, restricting our maximization problem to volume-stationary lattices is very simple since one have only to compare the lattices belonging to $\{\mathsf{A}_2,\Z^2,\Z^3,\mathsf{D}_3,\mathsf{D}_3^*\}$. Furthermore, Theorem \ref{thm:density} tells us that $\mathsf{A}_2,\Z^2,\Z^3,\mathsf{D}_3$ and $\mathsf{D_3}^*$ are critical points of $\mathcal{F}_\mu^\alpha$ for any $\alpha>0$ and $\mu$ radially symmetric on a ball $\Sigma=B_R$ for any $R>0$. In this case, combined with our numerical findings stated above (see also Section \ref{numerics}), we obtain new evidences that $\mathsf{A}_2$ and $\mathsf{D}_3$ are optimal for the grid cells problem if we restrict our study to volume-stationary lattices and radially symmetric firing field and measure $\mu$. More precisely, our numerics suggest that, for $\Sigma=B_R$ and $\mu=\sigma_R$ being the uniform measure on $B_R$, we have to following.
\begin{itemize}
\item For all fixed $R>0$, there exists $\alpha_R=\alpha(R,d)$ such that
\begin{itemize}
\item in dimension $d=2$, $\mathcal{F}_\mu^\alpha(\mathsf{A}_2)>\mathcal{F}_\mu^\alpha(\Z^2)$ (resp. $\mathcal{F}_\mu^\alpha(\mathsf{A}_2)<\mathcal{F}_\mu^\alpha(\Z^2)$) if $\alpha>\alpha_R$ (resp. if $\alpha<\alpha_R$);
\item in dimension $d=3$, for all $\alpha>\alpha_R$, $\mathcal{F}_\mu^\alpha(\mathsf{D}_3)>\mathcal{F}_\mu^\alpha(\mathsf{D}_3^*)>\mathcal{F}_\mu^\alpha(\Z^3)$. Furthermore, there exists $\tilde{\alpha}_R=\tilde{\alpha}(R)$ such that for all $\alpha<\alpha_R$, $\mathcal{F}_\mu^\alpha(\mathsf{D}_3)<\mathcal{F}_\mu^\alpha(\mathsf{D}_3^*)<\mathcal{F}_\mu^\alpha(\Z^3)$.
\end{itemize}
\item  For all $\alpha>\alpha_d$ for some $\alpha_d>0$,  there exists $R_\alpha=R(\alpha,d)$ such that
\begin{itemize}
\item in dimension $d=2$, $\mathcal{F}_\mu^\alpha(\mathsf{A}_2)>\mathcal{F}_\mu^\alpha(\Z^2)$ (resp. $\mathcal{F}_\mu^\alpha(\mathsf{A}_2)<\mathcal{F}_\mu^\alpha(\Z^2)$) if $R<R_\alpha$ (resp. if $R>R_\alpha$);
\item in dimension $d=3$, for all $R<R_\alpha$, $\mathcal{F}_\mu^\alpha(\mathsf{D}_3)>\mathcal{F}_\mu^\alpha(\mathsf{D}_3^*)>\mathcal{F}_\mu^\alpha(\Z^3)$. Furthermore, there exists $\tilde{R}_\alpha=\tilde{R}(\alpha)>0$ such that, for all $R>\tilde{R}_\alpha$, $\mathcal{F}_\mu^\alpha(\mathsf{D}_3)<\mathcal{F}_\mu^\alpha(\mathsf{D}_3^*)<\mathcal{F}_\mu^\alpha(\Z^3)$.
\end{itemize}
\end{itemize}

\subsection{Conclusion and open problems}

We conclude that, whereas the rigorous study of $\mathcal{F}_\mu^\alpha$ stays a very challenging problem, our asymptotic and numerical investigations show new evidences of the optimality of the triangular and face-centred-cubic lattices for the grid cells problem in dimension 2 and 3. In particular, our results suggest the following conjecture that also covers the experimental case where $R=0.16$ (see Remark \ref{rmk:Brun} on the size of the firing fields).

\begin{Conjecture}[\textbf{Maximality of the best packing in dimension 2 and 3}]
Let $\sigma_R$ be the uniform measure on $B_R$ and $\alpha=\frac{10}{\pi}$. Furthermore, let $R_2:=\frac{1}{2}\sqrt{\frac{2}{\sqrt{3}}}$ and $R_3:=2^{-\frac{5}{6}}$. Then we have:
\begin{enumerate}
\item For all $R\leq R_2$, $\mathsf{A}_2$ is the unique maximizer of $\mathcal{F}_{\sigma_R}^\alpha$ in $\mathcal{L}_2(1)$.
\item For all $R\leq R_3$, $\mathsf{D}_3$ is the unique maximizer of $\mathcal{F}_{\sigma_R}^\alpha$ in $\mathcal{L}_3(1)$.
\end{enumerate}
\end{Conjecture}
This problem by itself turns out to be a very interesting mathematical question with the need to derive new results for translated lattice theta functions. For instance, one may want to push further the properties of the Heat Kernel and to use the connection with $\mathcal{F}_\mu^\alpha$ explained in Remark \ref{rmk:heat}. 

\medskip

Moreover, we have also compared our model to the results found by Brun et al. in \cite{ProgressiveScale} about the size of the firing field $\Sigma$ related to the lattice spacing. It appears that the optimality properties of our model support the experimental observations of Brun et al. for a wide range of parameters $\alpha>1.25$. Also notice that in their recent work \cite{Nageleetal}, Nagele et al. suggest that the ratio between the diameter of $\Sigma$ (i.e. $2R$ if $\Sigma=B_R$) and the lattice spacing should be slightly smaller than the one found in \cite{ProgressiveScale}. Nevertheless, it seems that one could find again a range of $\alpha$ such that the triangular lattice again maximizes the Fisher Information's trace.

\medskip

For a future work, it would also be interesting to study the $\mu$-dependence of the maximizer of $\mathcal{F}_\mu^\alpha$, in particular when $\mu$ is not uniform nor radially symmetric. In the latter case, one could try to explain the existence of sheared grids observed for instance in \cite{Sheargridcells}, also for grid cells when for example the animal visits the boundary of its displacement zone. Another direction or research would also be to replace the Gaussian function by another tuning curve and to study the same type of optimality problem. Finally, we think that our results and observation might be useful for other biological or physical models where periodic Fisher Information with Poisson distribution are involved.

\medskip

\textbf{Plan of the paper.} In Section \ref{sec:Fisher}, we explain how we obtain formula \eqref{eq:Fmualpha} and on which mathematical assumptions we have built our model. Furthermore, the scaling formula \eqref{eq:scalingintro} is shown in Proposition \ref{prop:scale}. Section \ref{sec:alternative} is devoted to the proof of Theorem \ref{thm:alpha0epsilon0} and to the alternative formula \ref{eq:altF}. The characterization of volume-stationary lattices is done in Section \ref{sec:volstat} where the proof of Theorem \ref{thm:density} is given. Finally, our numerical investigations are presented in Section \ref{numerics} where we first discuss the parametrization of lattices as well as the properties of $\mathcal{Q}_L^\alpha(y)$.

\section{Description of the model, derivation of formula \eqref{eq:Fmualpha} and scaling}\label{sec:Fisher}

In this section, we are justifying formula \eqref{eq:Fmualpha} by listing all our assumptions. As explained in the introduction, we are following Mathis et al. \cite{Mathisetal} for building our model, with only small modifications.

\medskip

We consider neurons (called grid cells) that are firing (spiking) according to the position (the stimulus) $x\in \R^d$ of an animal. We are also considering an arbitrary fixed time interval $[0,T]$ where the neurons are spiking. We are assuming the following hypothesis \textbf{(A1)-(A6)} where the Gaussian is replaced for the moment by an arbitrary function $f$ as in \eqref{eq:Ef}.
\begin{itemize}
\item[\textbf{(A1)}] \textbf{Independence of neurons.} The neurons we consider are independently firing.
\end{itemize}
Given $n$ independent neurons such that the $i^{th}$ neuron fires $k_i$ times in the time interval $[0,T]$, the probability of having $K=(k_1,...,k_n)$ spikes when an animal is at position $x\in \R^d$ is 
\begin{equation}\label{eq:probaP}
P(K| x)=\prod_{i=1}^n P_i(k_i |x),
\end{equation}
where $P_i(k_i|x)$ is the probability of firing the $i^{th}$ neuron $k_i$ times when the animal is at $x$.

\begin{itemize}
\item[\textbf{(A2)}] \textbf{Firing Poisson statistics.} The neuron's firing follows a Poisson distribution. More precisely, we assume that
\begin{equation}
P_i(k_i|x)=\frac{(\Omega_i(x))^{k_i} e^{-\Omega_i(x)}}{k_i !},
\end{equation}
where the tuning curve $\Omega_i(x)$ is the average firing number of the $i^{th}$ neuron at position $x$ in the time interval $[0,T]$. 
\end{itemize}

\begin{itemize}
\item[\textbf{(A3)}] \textbf{Spread out lattice-periodic tuning curve shape.} All the tuning curves of the neurons are identically equal to 
\begin{equation}
\Omega_i(x)=\Omega_L(x)=E_f[L+x]:=\sum_{p\in L} f(|p+x|^2),\quad \forall i\in \{1,..., n\},
\end{equation}
for a lattice $L\in \mathcal{L}_d(1)$ and a function $f:[0,\infty)\to \R$ as in \eqref{eq:Ef}, i.e.
$$
 |f(r)|=O(r^{-\frac{d}{2}-\eta}), \eta>0\quad \textnormal{as $r\to +\infty$},
$$
in order for $\Omega_L(x)$ to be absolutely convergent. We will call $\Omega_L:\R^d\to \R_+$ the grid cell's tuning curve. When $f$ has its maximum at $r=0$, it means that there is more spikes when the animal cross a lattice site, which is the usual assumptions made on $f$. This tuning curve's shape is slightly different than the periodic boundary conditions assumed by Mathis et al. in \cite{Mathisetal}.
\end{itemize}

We now define a grid module as an ensemble of $N\in \N$ shifted grid cells $\{ \Omega_{L+y_j}\}_{j=1}^N$ where the vectors $Y=(y_j)_j\subset \R^d$ are called the spatial phases of the module. Furthermore, we have for all $x\in \R^d$, all lattice $L\in \mathcal{L}_d(1)$ and all shifting vector $y_j\in Y$,
\begin{equation}
\Omega_{L+y_j}(x)=E_f[L+y_j+x]:=\sum_{p\in L} f(|p+x+y_j|^2).
\end{equation}

Therefore, for any spatial phase with shifting vector $y_j$, we write the probability in \eqref{eq:probaP} as
\begin{equation}
P_j(K |x)=e^{-n \Omega_{L+y_j}(x)}\prod_{i=1}^n \frac{(\Omega_{L+y_j}(x))^{k_i}}{k_i !},\quad j\in \{1,...,N\}.
\end{equation}

As explained in \cite{Mathisetal}, we are interested in solving the following mathematical question: given a spike count vector $K=(k_1,...,k_n)$, where is the animal? The estimation of this position is written $\widehat{x}(K)$. We therefore want to minimize the square of the error in the decoding process which is given by
\begin{equation}
\varepsilon(x | \widehat{x}):= \mathbb{E}_{P(K |x)} \left(\| x- \widehat{x}(K)\|^2  \right).
\end{equation}

Minimizing the error $\varepsilon(x | \widehat{x})$ is the same as minimizing the trace of the covariance matrix $\textnormal{Cov}(x,\widehat{x})$. Furthermore, the Fisher Information (see e.g. \cite[Eq. (11.291)]{CoverThpas}) associated to the probability $P$ with Poisson distribution is defined by the matrix $J(x)=(J_{\ell,m}(x))_{\ell,m}$ where
\begin{equation}
J_{\ell,m}(x):=\int \left[\partial_{x_\ell} \log P(K |x)  \partial_{x_m} \log P(K |x)\right] P(K|x)dK.
\end{equation}
\begin{itemize}
\item[\textbf{(A4)}] \textbf{Unbiased estimator assumption.} We assume that the Cramer-Rao lower bound holds (see e.g. \cite[Thm. 11.10.1]{CoverThpas}), i.e. $\textnormal{Cov}(x,\widehat{x})\geq J(x)^{-1}$. This is actually automatically satisfied in the independent Poisson process case.
\end{itemize}

For the grid module, (A1) ensures that the Fisher Information is the sum of Fisher Information for all the $N$ spatially shifted phases $Y$, i.e.

\begin{equation}
J_{\ell,m}(x)=\sum_{j=1}^N \int \left[\partial_{x_\ell} \log P_j(K |x)  \partial_{x_m} \log P_j(K |x)\right] P_j(K|x)dK.
\end{equation}

A straightforward computation gives
\begin{equation}
\mathrm{Tr}J(x)= \sum_{\ell=1}^d \sum_{j=1}^N \left(\partial_{x_\ell} \Omega_{L+y_j}(x)\right)^2  \int \left(-n + \frac{ \sum_{i=1}^n k_i}{\displaystyle \Omega_{L+y_j}(x)}  \right)^2 P_j(K | x)dK.
\end{equation}
The integral is in fact the Fisher Information of $n$ independent Poisson processes with the same parameter $\lambda= \Omega_{L+y_j}(x)$, which means that it is $n$ times the Fisher Information of one single Poisson process with the same parameter (see e.g. \cite[p. 395]{CoverThpas}). Therefore, for all $y_j\in Y$,
\begin{equation}
\int \left(-n + \frac{ \sum_{i=1}^n k_i}{\displaystyle \Omega_{L+y_j}(x)}  \right)^2 P_j(K | x)dK=\frac{n}{\Omega_{L+y_j}(x)}.
\end{equation}
We therefore obtain that
\begin{equation}\label{tracefinal}
\mathrm{Tr}J(x)= 
n\sum_{j=1}^N \mathcal{Q}_L(y_j+x),\quad \textnormal{where}\quad \mathcal{Q}_{L}(y):= \frac{\left|\nabla_x E_f[L+y]\right|^2}{E_f[L+y]}.
\end{equation}

Optimizing such lattice energy is a huge challenge since the maximizer varies a lot with $x$ and $Y$ (see Section \ref{subsec:disc}). That is why it is not appropriate to our study. Therefore, we are again following \cite{Mathisetal} in order to transform the discrete set of shifts $Y$ into a continuous one written $\Sigma$.

\begin{itemize}
\item[\textbf{(A5)}] \textbf{Aggregation of spatial phases.} We assume that $\displaystyle \frac{1}{N}\sum_{j=1}^N \delta_{y_j}\rightharpoonup\mu \in \mathcal{M}(\Sigma)$ as $N\to \infty$ in the weak-$\ast$ sense, with support in a compact set $\Sigma\subset \R^2$. This is suggested by experiments (see \cite{Haftingetal}). The set $\Sigma$ can be considered as the support of the grid cells module, also called the firing field of the module. Actually, there is a scale dependence between $\Sigma$ and $L$ in the sense that, if the distances in the lattice are multiplied by a real number $\lambda$, then it is the same for $\Sigma$ (see e.g. \cite{Quirk,MECdiscrete} and Proposition \ref{prop:scale}). Furthermore, it has been observed in \cite{ProgressiveScale} that the smallest width of $\Sigma$ is 30cm-50cm for a lattice spacing (corresponding to the triangular lattice) of 50cm-1m (see also \cite{Nageleetal} for other data with smaller firing fields width compared to lattice spacing), i.e. firing fields are not overlapping (as in Figure \ref{fig:tri}). We also notice that, presented as above, $\mu$ should be a probability measure. However, in the following we are considering a more general Radon measure.
\end{itemize}
Therefore, we want to maximize the following Fisher Information per neuron for the module $M=(L,\Omega_L^\alpha,\mu,\Sigma)$, where the limit is justified by the fact that $y\mapsto \mathcal{Q}_L(y)$ is continuous and bounded on the compact set $\Sigma$,
\begin{equation}
J_M(x):=\frac{1}{n}\lim_{N\to \infty}\frac{\mathrm{Tr}J(x)}{  N}=\int_{\Sigma} \frac{\left| \nabla_x \Omega_{L+y}(x)\right|^2}{\Omega_{L+y}(x)}d\mu(y)=\int_\Sigma \frac{\left|\nabla_x E_f[L+y+x] \right|^2}{E_f[L+y+x]}d\mu(y).
\end{equation}

\begin{itemize}
\item[\textbf{(A6)}] \textbf{Re-centering of the problem.} We assume that $x=0$. Thus, the problem is simplified to maximize the resolution at $x=0$ of the decoding process.
\end{itemize}
Therefore, given $f$ satisfying assumptions \eqref{eq:Ef} and $\mu\in \mathcal{M}(\Sigma)$, we want to maximize among lattices $L$ the following functional:
\begin{equation}
\mathcal{F}_{\mu}(L)=J_M(0)=\int_\Sigma \frac{\left|\nabla_y E_f[L+y] \right|^2}{E_f[L+y]}d\mu(y).
\end{equation}

Furthermore, if $f\geq 0$, then we have the following more compact form 
\begin{equation}\label{eq:formulaFdetailed}
\mathcal{F}_{\mu}(L)=\int_\Sigma \left| \nabla_y \sqrt{E_f[L+y]}\right|^2 d\mu(y)
\end{equation}
which gives us \eqref{eq:Fmualpha} by choosing $f(r)=e^{-\pi \alpha r}$ and renaming the functional $\mathcal{F}_\mu^\alpha$ to show the $\alpha$-dependence. 

\medskip

For the reader's convenience, the elementary form of $\mathcal{F}_\mu^\alpha(L)$ is
$$
\mathcal{F}_\mu^\alpha(L)=4\pi^2 \alpha^2 \displaystyle \bigints_\Sigma \frac{\displaystyle \sum_{i=1}^d \left(\sum_{p\in L} (p_i+y_i) e^{-\pi \alpha |p+y|^2}  \right)^2}{\displaystyle \sum_{p\in L} e^{-\pi \alpha |p+y|^2}}d\mu(y).
$$

\begin{remark}[Fisher Information in terms of the Heat Kernel]\label{rmk:heat}
It has to be noticed that $\mathcal{F}_\mu^\alpha$ can be expressed in terms of the Heat Kernel on $L$. More precisely, and as we already pointed out in \cite{Baernstein-1997,BeterminKnuepfer-preprint},  let $u_L(y,t)$ be the temperature at point $y$
  and at time $t$, if at initial time $t=0$ a heat source of unit strength is placed at
  each point of $L$, i.e. $u_L$ is defined, for any lattice $L\subset \R^d$, any $y\in \R^d$  and any $t>0$ as the solution of
\begin{align*}
  \left \{ %
  \begin{array}{ll}
    \partial_t u_L(y,t) = \Delta_y u_L \quad\quad\quad\quad &\text{for  $(y,t) \in \R^d \times (0,\infty)$} \vspace{0.3ex} \\
    u_L(y,0) = \sum_{p \in L} \delta_{p} \quad\quad\quad\quad &\text{for $y \in \R^d$},
  \end{array}
  \right.
\end{align*}
where $\delta_p$ is the Dirac measure at $p \in \R^d$. Then, for any $\alpha>0$, any lattice $L\subset \R^d$ and any measure $\mu\in \mathcal{M}(\Sigma)$,
\begin{equation}
\mathcal{F}_\mu^\alpha(L)=\frac{1}{\alpha^{\frac{d}{2}}}\int_\Sigma \left|\nabla_y \sqrt{u_L\left(y,\frac{1}{4\pi \alpha}  \right)}\right|^2d\mu(y).
\end{equation}
\end{remark}

We now show the scaling formula already known by Mathis et al. in \cite[Eq (22)]{Mathisetal} and applied to the Gaussian periodic tuning curve. It is crucial to remember that, as recalled above and according to \cite{Quirk,MECdiscrete}, changing the density of the lattice $L$ also changes the size of the firing field $\Sigma$ accordingly.

\begin{prop}[\textbf{Scaling formula}]\label{prop:scale}
For any $\lambda>0$, any $\mu\in \mathcal{M}(\Sigma)$, any $L\in \mathcal{L}_d(1)$ and any $\alpha>0$, we have
\begin{equation}
\mathcal{F}^\alpha_{\mu_\lambda}(\lambda L)=\lambda^{-2}\mathcal{F}^{\lambda^2 \alpha} _{\mu}(L).
\end{equation}
\end{prop}
\begin{proof}
By a simple change of variable, we have
\begin{align*}
\mathcal{F}^\alpha _{\mu_\lambda}(\lambda L)&=\int_{\lambda \Sigma} \left| \nabla_y \sqrt{\theta_{\lambda L+y}(\alpha)}\right|^2 d\mu_\lambda(y)\\
&=\frac{1}{\lambda^2}\int_\Sigma \left| \nabla_y \sqrt{\theta_{\lambda L+\lambda y}(\alpha)}\right|^2 d\mu(y)\\
&=\frac{1}{\lambda^2}\int_\Sigma  \left| \nabla_y \sqrt{\theta_{ L+y}(\lambda^2\alpha)}\right|^2d\mu(y)=\lambda^{-2}\mathcal{F}^{\lambda^2 \alpha} _{\mu}(L).
\end{align*}
\end{proof}

Since for any $\alpha>0$, any $\mu\in \mathcal{M}(\Sigma)$ and any $L\in \mathcal{L}_d(1)$, $\lim_{\lambda\to 0}\mathcal{F}^\alpha _{\mu_\lambda}(\lambda L)= \lim_{\lambda \to 0} \lambda^{-2}\mathcal{F}^{\lambda^2 \alpha} _{\mu}(L)=+\infty$, it is sufficient to look for the maximizer of $\mathcal{F}_\mu^\alpha$ for fixed $\alpha$, $\mu\in \mathcal{M}(\Sigma)$ and  $L\in \mathcal{L}_d(1)$. If the density is not fixed, then it is then sufficient to take the high density limit to find $+\infty$ as the maximum of $\mathcal{F}_\mu^\alpha$. Notice that we also have, for any $\alpha>0$, any $\mu\in \mathcal{M}(\Sigma)$ and $L\in \mathcal{L}_d(1)$, $\lim_{\lambda\to \infty}\mathcal{F}^\alpha _{\mu_\lambda}(\lambda L)= \lim_{\lambda \to \infty} \lambda^{-2}\mathcal{F}^{\lambda^2 \alpha} _{\mu}(L)=0$.



%

\section{Alternative formula and proof of Theorem \ref{thm:alpha0epsilon0}}\label{sec:alternative}
We are proving Theorem \ref{thm:alpha0epsilon0} in this part, using an alternative formula given in \cite{Regev:2015kq}.

\begin{prop}[\textbf{Alternative formula for $\mathcal{F}_\mu^\alpha$}]
For any $\alpha>0$, any $\mu\in \mathcal{M}(\Sigma)$ and any $L\in \mathcal{L}_d(1)$, 
\begin{equation}\label{eq:altF}
\mathcal{F}^\alpha_{\mu}(L)=4\pi^2 \alpha^2 \int_\Sigma \theta_{L+y}(\alpha) \left| \underset{w \sim D_{L+y,\alpha}}{\mathbb{E}[w]}\right|^2 d\mu(y),
\end{equation}
where $D_{L+x,\alpha}$ is the discrete Gaussian distribution over $L+x$ with parameter $\alpha$, i.e. the probability distribution over $L+x$ that assigns probability proportional to $e^{-\pi \alpha |w|^2}$ to each vector $w\in L+x$.
\end{prop}

\begin{remark}
Let us specify that the above expected value is actually the following vector:
\begin{equation}\label{eq:Edetailed}
\underset{w \sim D_{L+y,\alpha}}{\mathbb{E}[w]}=\frac{1}{\theta_{L+y}(\alpha)}\sum_{p\in L} (p+y) e^{-\pi \alpha |p+y|^2}.
\end{equation}
\end{remark}

\begin{proof}
Following \cite[p. 6]{Regev:2015kq} (with $s=\alpha^{-1/2}$), a straightforward computation gives (see \eqref{eq:Edetailed})
\begin{equation}
\frac{\nabla_x \theta_{L+x}(\alpha)}{\theta_{L+x}(\alpha)}=-2\pi \alpha \underset{w \sim D_{L+x,\alpha}}{\mathbb{E}[w]}.
\end{equation}
 It follows that
\begin{equation}
\frac{|\nabla_y \theta_{L+y}(\alpha)|^2}{\theta_{L+y}(\alpha)}=4\pi^2 \alpha^2 \theta_{L+y}(\alpha) \left| \underset{w \sim D_{L+y,\alpha}}{\mathbb{E}[w]}\right|^2,
\end{equation}
which shows \eqref{eq:altF}.
\end{proof}

%
We now show Theorem \ref{thm:alpha0epsilon0} by directly applying \eqref{eq:altF}.

\begin{proof}[Proof of Theorem \ref{thm:alpha0epsilon0}]
A straightforward computation shows that, for any $\mu\in \mathcal{M}(\Sigma)$ and any $\lambda>0$, as $\alpha\to 0$,
\begin{equation}
\mathcal{F}^\alpha_{\mu_\lambda}(L)=4\pi^2 \alpha^2 \int_{\Sigma} \theta_{L+y}(\alpha) \left| \underset{w \sim D_{L+y,\alpha}}{\mathbb{E}[w]}\right|^2 d\mu_\lambda(y)\sim 4\pi^2 \alpha^2 \int_{\Sigma} \theta_{L+y}(\alpha)  d\mu_\lambda(y),
\end{equation}
which is uniquely minimized by $\mathsf{A}_2$ (resp. locally minimized by $\mathsf{D}_3^*$) by a direct application of \cite[Thm. 2 and 3]{BetKnupfdiffuse} (resp. \cite[Thm. 4.3]{BetSoftTheta}) as $\lambda$ is small enough or if $\mu$ has the form $d\mu(x)=\rho(|x|^2)dx$ and $\rho$ is a completely monotone function. The fact that $\mathcal{F}^\alpha_{\mu_\lambda}$ does not have any maximizer as $\alpha,\lambda$ are small enough is again a simple consequence of our work in \cite{BetSoftTheta} and the fact that the lattice theta function $L\mapsto \theta_L(\alpha)$ does not have a maximizer in $\mathcal{L}_d(1)$. Indeed, it is enough to take a sequence of orthorhombic lattices with degenerate shapes for showing that the theta functions goes to infinity.
\end{proof}
\begin{remark}
The same is true in dimension $d\in \{8,24\}$ replacing $\mathsf{A}_2$ by $\mathsf{E}_8$ or $\Lambda_{24}$, again by an application of \cite[Thm 4.5]{BetSoftTheta} and the universal optimality of these lattices shown in \cite{CKMRV2Theta}.
\end{remark}

\section{Volume-stationary lattices - Proof of Theorem \ref{thm:density}}\label{sec:volstat}
We now show Theorem \ref{thm:density} by using a combination of Proposition \ref{prop:scale}, formula \eqref{eq:altF} and our works in \cite{BetKnupfdiffuse,LBMorse}.
\begin{proof}[Proof of Theorem \ref{thm:density}]
 We consider $\mathcal{F}^\alpha_{\mu_\lambda}(\lambda L)$ for $\lambda>0$, then by assumption we have that $\nabla_L \mathcal{F}^\alpha_{\mu_\lambda}(\lambda L)=0$ for all $\lambda\in I$, where $I$ is an open interval of $\R_+$. We therefore have, for all $\lambda\in I$,
\begin{equation}
\nabla_L \int_{\Sigma} \left| \nabla_y \sqrt{\theta_{\lambda L+y}(\alpha)}\right|^2 d\mu(y)=0,
\end{equation}
the gradient $\nabla_L$ being taken in $\mathcal{L}_d(1)$. Since $\lambda\mapsto \nabla_L \mathcal{F}^\alpha_{\mu}(\lambda L)$ is analytic by integration on a compact set and absolute convergence, its zero are therefore isolated and it follows that necessarily $I=(0,\infty)$. We now use formula \eqref{eq:altF} combined with Proposition \ref{prop:scale} and we get
\begin{align*}
\mathcal{F}_{\mu_\lambda}^\alpha(\lambda L)=\lambda^{-2} \mathcal{F}_\mu^{\lambda^2 \alpha}(L)= 4\pi^2 \lambda^2 \alpha^2 \int_\Sigma \theta_{L+y}(\lambda^2 \alpha)  \left| \underset{w \sim D_{L+y,\lambda^2\alpha}}{\mathbb{E}[w]}\right|^2 d\mu(y).
\end{align*}
Therefore, since the above squared expectation converges to a constant for small $\lambda^2 \alpha$, we obtain for $\lambda<\lambda_0$ sufficiently small,
\begin{equation}
 \nabla_L \left\{\int_\Sigma \theta_{L+y}(\lambda^2 \alpha) d\mu(y)\right\}=0.
\end{equation}
By Poisson Summation Formula (see e.g. \cite[Eq. (43)]{ConSloanPacking}), we obtain as in \cite{BetKnupfdiffuse,BetSoftTheta} that, for all $\lambda,\alpha >0$,
\begin{align*}
\int_\Sigma \theta_{L+y}(\lambda^2 \alpha) d\mu(y)=\int_\Sigma \frac{1}{(\lambda^2 \alpha)^{\frac{d}{2}}}\sum_{q\in L^*} e^{-\frac{\pi |q|^2}{\lambda^2 \alpha}} e^{2i\pi q\cdot y}d\mu(y)&=\frac{1}{(\lambda^2 \alpha)^{\frac{d}{2}}}\sum_{q\in L^*} e^{-\frac{\pi |q|^2}{\lambda^2 \alpha}} \int_\Sigma e^{2i\pi q\cdot y}d\mu(y)\\
&=\frac{1}{(\lambda^2 \alpha)^{\frac{d}{2}}}\sum_{q\in L^*} e^{-\frac{\pi |q|^2}{\lambda^2 \alpha}} \hat{\mu}(q),
\end{align*}
where $\hat{\mu}$ is the Fourier transform of the measure $\mu$. By analyticity of $\lambda\mapsto \nabla_L \int_\Sigma \theta_{L+y}(\lambda^2 \alpha) d\mu(y)$, we therefore get again that, for all $\lambda>0$,
\begin{equation}\label{eq:critic}
\nabla_L \sum_{q\in L^*} e^{-\frac{\pi}{\alpha \lambda^2}|q|^2}\widehat{\mu}(q)=0.
\end{equation}

Since $\mu$ is radially symmetric, then its Fourier transform is radially symmetric, i.e. $\widehat{\mu}(q)=h(|q|)$ for some function $h:\R_+\to \R$. Now, the same argument can be used as in \cite{DeloneRysh,Gruber,LBMorse} to conclude that the first layer of $L^*$ must be strongly eutactic. Indeed, \eqref{eq:critic} implies that, for all $t>0$ (which could be written $t=\lambda^2 \alpha$),
$$
\nabla_{L^*} \theta_{L^*}\left( \frac{1}{t}\right)=0,
$$
which implies that
$$
\nabla_{L^*} \zeta_{L^*}(s)=0
$$
for all $s>d$ for the Epstein zeta function defined by
$$
\zeta_{L^*}(s):=\sum_{q\in L^*\backslash \{0\}} \frac{1}{|q|^s}=\int_0^\infty \left(\theta_{L^*}\left( \frac{t}{\pi}\right)-1\right)\frac{t^{\frac{s}{2}-1}}{\Gamma(s/2)}dt,
$$
where the second equality follows from a simple Laplace transform argument (see e.g. \cite{BetTheta15}). We therefore apply \cite[Corollary 1]{Gruber} to conclude that all the layers of $L^*$ are necessarily strongly eutactic. In dimension $d\in \{2,3\}$, all these lattices are characterized (see for instance \cite[Corollary 2]{Gruber}) and they are the one listed in the theorem's statement for which the list of duals stays unchanged. This completes the proof.
\end{proof}
\begin{remark}
The proof still holds if the Gaussian is replaced by a function $f$ which is the Laplace transform of a measure. We also suspect that the result is also true for some non-radially symmetric measure $\mu$.
\end{remark}

\section{Numerical investigation}\label{numerics}
We are using the software Scilab to perform our numerical investigation. All the values are computed by using a sufficient number of points for estimating our infinite sums according to the value of $\alpha$ we choose (mostly $\alpha=\frac{10}{\pi}\approx 3.18$).

\subsection{Parametrization of lattices in dimension 2 and 3}

We recall here how to parametrize a $d$-dimensional lattice in order to perform our numerical computations. We basically follow the lines of \cite{Beterloc,Beterminlocal3d} by stating that a two-dimensional (resp. three-dimensional) lattice of unit density can be parametrized by two real numbers $(x,y)$ (resp. five real numbers $(u,v,x,y,z)$) belonging to a fundamental domain $\mathcal{D}_2\subset \R^2$ (resp. $\mathcal{D}_3\subset \R^5$) where only one copy of each lattice appears (see e.g. \cite[Sect. 1.4]{Terras_1988}). More precisely, two-dimensional and three-dimensional lattices of unit density can be written respectively as
\begin{align*}
&L=\Z\left(\frac{1}{\sqrt{y}},0  \right)\oplus \Z \left( \frac{x}{\sqrt{y}}, \sqrt{y} \right)\in \mathcal{L}_2(1)\\
&L=2^{\frac{1}{6}}\left[ \Z\left( \frac{1}{\sqrt{u}},0,0 \right)\oplus \Z\left( \frac{x}{\sqrt{u}},\frac{v}{\sqrt{u}},0 \right)\oplus \Z \left(\frac{y}{\sqrt{u}},\frac{v z}{\sqrt{u}},\frac{u}{v\sqrt{2}}  \right) \right]\in \mathcal{L}_3(1).
\end{align*}
In dimension $d=2$, this fundamental domain $\mathcal{D}_2$ is easy to describe whereas it is more complicated in dimension $d=3$. In particular, we have
$$
\mathcal{D}_2:=\left\{ (x,y)\in \R^d : 0\leq x\leq \frac{1}{2}, y>0, x^2+y^2\geq 1 \right\}.
$$
The square lattice $\Z^2$ and triangular lattice $\mathsf{A}_2$ are respectively represented by the points $(0,1)\in \mathcal{D}_2$ and $\left(\frac{1}{2},\frac{\sqrt{3}}{2}  \right)\in \mathcal{D}_2$. We omit to write what an example of fundamental domain of $\mathcal{L}_3(1)$ could be, but we recall that the simple cubic lattice $\Z^3$, the FCC lattice $\mathsf{D}_3$ and the BCC lattice $\mathsf{D}_3^*$ are respectively represented by $(2^{1/3},1,0,0,0)\in \mathcal{D}_3$, $(1,1,0,1/2,1/2)\in \mathcal{D}_3$ and $(2^{-1/3},1,0,1/2,1/2)\in \mathcal{D}_3$. We recall that, since the three-dimensional case is too demanding in terms of computational time, we will only compare certain lattices and compute the Hessian matrices at the FCC lattice.

%
%
%

\subsection{The pure discrete case: combinations of $\mathcal{Q}_L(y)$}\label{subsec:disc}
We first explain why the problem of maximizing  the discrete version of the Fisher Information's trace \eqref{tracefinal} (choosing again $x=0$), i.e.
$$
\mathrm{Tr}J(0)= n\sum_{j=1}^N \mathcal{Q}_L(y_j),\quad \mathcal{Q}_L(y):=\frac{\left|\nabla_y \theta_{L+y}(\alpha)\right|^2}{\theta_{L+y}(\alpha)},
$$
is very difficult and therefore why it is convenient to consider -- as suggested by experiments \cite{Haftingetal,Gridcellsreview,Nageleetal} -- a continuous firing field $\Sigma$ instead of a set of points $Y$. We only consider the two-dimensional case, but the same remarks hold for higher dimensions.

\medskip

For any $L=\bigoplus_{i=1}^d \Z u_i$, if for all $j\in\{1,...,N\}$, $y_j\in \{u_i, \frac{u_i}{2}, \frac{1}{2}\sum_{i} u_i\}$ for any $i\in \{1,...,d\}$, then $\mathrm{Tr}J(0)=0$ since these points $y_j$ are known to be critical points for $y\mapsto \theta_{L+y}(\alpha)$ -- and more generally $y\mapsto E_f[L+Y]$ for any function $f$ satisfying \eqref{eq:Ef}. They are the trivial zeros of $\mathrm{Tr}J(0)$. In the particular case $L=\mathsf{A}_2$, the barycenters of the primitive triangles are also critical points of  $y\mapsto \theta_{L+y}(\alpha)$ for all $\alpha>0$ as shown in \cite{Baernstein-1997} (see Figure \ref{Gradient3}). Furthermore, in Figure \ref{fig:exdiscrete}, we show that the maximizer of $\mathrm{Tr}J(0)$ varies a lot with $\{y_j\}_j$.

\medskip

\begin{figure}[!h]
\centering
\includegraphics[width=8cm]{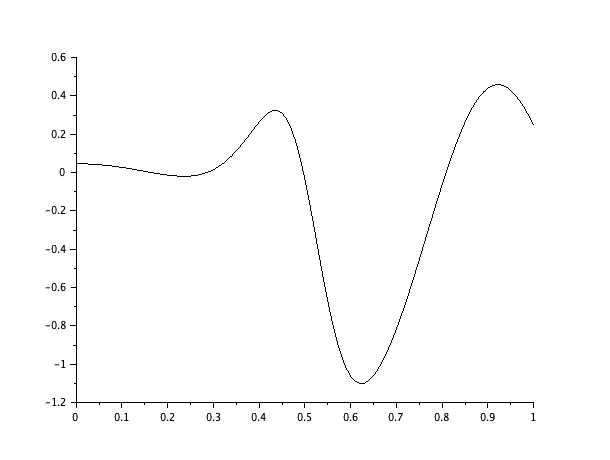} 
\caption{Plot of $t\mapsto \mathcal{Q}^\alpha_{\mathsf{A}_2}(0.1,0.1)+\mathcal{Q}^\alpha_{\mathsf{A}_2}(t,0.2)-\left(\mathcal{Q}^\alpha_{\Z^2}(0.1,0.1)+\mathcal{Q}^\alpha_{\Z^2}(t,0.2)\right)$ for $\alpha=\frac{10}{\pi}$ and $t\in [0,1]$, showing that the maximizer of $\mathrm{Tr}J_M(0)$ is not always triangular.}
\label{fig:exdiscrete}
\end{figure}

\begin{figure}[!h]
\centering
\includegraphics[width=5cm]{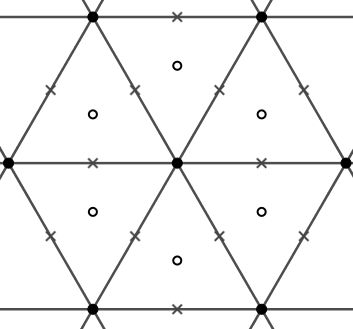} 
\caption{Zeros of $x\mapsto \nabla_x\theta_{\mathsf{A}_2+x}(\alpha)$ in a patch of $\mathsf{A}_2$. Following \cite{Baernstein-1997}, $\bullet$ are the maximizers, $\circ$ the minimizers and $\times$ the saddle points of the function $x\mapsto \theta_{\mathsf{A}_2+x}(\alpha)$. These are the only critical points of this function.}
\label{Gradient3}
\end{figure}

It has also to be noticed that $L\mapsto \displaystyle \mathcal{Q}^\alpha_{L}(y)$ alone is not always maximized by the same lattice as $y$ varies in $\Sigma=B_R$. Indeed, writing $y=(r\cos \theta, r\sin \theta)$ and fixing $r$, we compare the energies of the square and triangular lattices by plotting
\begin{equation}\label{eq:gr}
\theta\mapsto g_r(\theta):=\mathcal{Q}_{\Z^2}^{\alpha}(r\cos\theta, r\sin \theta)-\mathcal{Q}_{\mathsf{A}_2}^{\alpha}(r\cos\theta, r\sin \theta),\quad \alpha=\frac{10}{\pi}.    
\end{equation}
It is easy to give an explicit formula for $g_r(\theta)$ since this is the difference of
\begin{align*}
\mathcal{Q}_{\Z^2}^{\alpha}(r\cos\theta, r\sin \theta)=&4\pi^2 \alpha^2\frac{\displaystyle \left(\sum_{m,n\in \Z}(m+r\cos \theta) e^{-\pi \alpha [(m+r\cos \theta)^2+(n+r\sin \theta)^2  ]} \right)^2}{\displaystyle \sum_{m,n\in \Z}  e^{-\pi \alpha [(m+r\cos \theta)^2+(n+r\sin \theta)^2]} }\\
&\quad\quad\quad+4\pi^2 \alpha^2\frac{\displaystyle \left( \sum_{m,n\in \Z}(n+r\sin\theta) e^{-\pi \alpha [(m+r\cos \theta)^2+(n+r\sin \theta)^2  ]} \right)^2}{\displaystyle \sum_{m,n\in \Z}  e^{-\pi \alpha [(m+r\cos \theta)^2+(n+r\sin \theta)^2]}}
\end{align*}
and
\begin{align*}
\mathcal{Q}_{\mathsf{A}_2}^{\alpha}(r\cos\theta, r\sin \theta)=&4\pi^2 \alpha^2\frac{\displaystyle \left(\sum_{m,n\in \Z}\left(\lambda(m+\frac{n}{2})+r\cos \theta\right) e^{-\pi \alpha [(\lambda(m+\frac{n}{2})+r\cos \theta)^2+(\frac{\lambda \sqrt{3}}{2}n+r\sin \theta)^2  ]} \right)^2}{\displaystyle \sum_{m,n\in \Z} e^{-\pi \alpha [(\lambda(m+\frac{n}{2})+r\cos \theta)^2+(\frac{\lambda \sqrt{3}}{2}n+r\sin \theta)^2  ]} }\\
&\quad\quad\quad+4\pi^2 \alpha^2\frac{\displaystyle \left( \sum_{m,n\in \Z}\left(\frac{\lambda \sqrt{3}}{2}n+r\sin\theta\right) e^{-\pi \alpha [(\lambda(m+\frac{n}{2})+r\cos \theta)^2+(\frac{\lambda \sqrt{3}}{2}n+r\sin \theta)^2  ]}\right)^2}{\displaystyle \sum_{m,n\in \Z}  e^{-\pi \alpha [(\lambda(m+\frac{n}{2})+r\cos \theta)^2+(\frac{\lambda \sqrt{3}}{2}n+r\sin \theta)^2  ]}}.
\end{align*}

Our observations are the following (see Figure \ref{Plotgr1}):
\begin{itemize}
\item For small $r$, we observe that $g_r(\theta)<0$ for all $\theta\in [0,2\pi]$. The triangular lattice has higher energy than the square lattice.
\item For $r>0.24$, there are alternation of signs for $g_r$ on some intervals of $[0,2\pi]$.
\item For $r=0.6$, the positive part of $g_r$ is more important than the negative one.
\end{itemize}
Therefore, showing the maximality of $\mathsf{A}_2$ for the integral of $\mathcal{Q}_L^\alpha$ on a ball of radius larger than $R_0=0.24$ is necessarily tricky. Furthermore, it is not surprising that the triangular lattice is almost never a critical point of the integrand in $\mathcal{L}_2(1)$. To illustrate this fact, we have plotted in Figure \ref{Gradient1} the function $\theta\mapsto | \nabla_{L}\mathcal{Q}^\alpha_{L}(r\cos \theta, r\sin \theta) |^2$ on $[0,2\pi]$ for $L=\mathsf{A}_2$, $\alpha=10/\pi$ and fixed $r\in \{0.2,0.3\}$. Notice that it is straightforward to compute $ | \nabla_{L}\mathcal{Q}^\alpha_{L}(r\cos \theta, r\sin \theta) |^2$, recalling that $\mathsf{A}_2$ corresponds to the point $(x,y)=(1/2,\sqrt{3}/2)$, since
\begin{align*}
&| \nabla_{L}\mathcal{Q}^\alpha_{L}(r\cos \theta, r\sin \theta)|_{L=\mathsf{A}_2} |^2\\
&=\left( \partial_x \mathcal{Q}^\alpha_{L}(r\cos \theta, r\sin \theta)|_{(x,y)=(1/2,\sqrt{3}/2)}\right)^2+\left( \partial_y \mathcal{Q}^\alpha_{L}(r\cos \theta, r\sin \theta)|_{(x,y)=(1/2,\sqrt{3}/2)}\right)^2,
\end{align*}
where
\begin{align*}
\mathcal{Q}^\alpha_{L}(r\cos \theta, r\sin \theta)&=4\pi^2 \alpha^2\frac{\displaystyle \left(\sum_{m,n\in \Z} \left( \frac{(m+xn)}{\sqrt{y}}+r\cos \theta \right) e^{-\pi \alpha \left[ \left( \frac{(m+xn)}{\sqrt{y}}+r\cos \theta \right)^2 + (n\sqrt{y}+r\sin \theta)^2 \right]} \right)^2}{\displaystyle \sum_{m,n\in \Z} e^{-\pi \alpha \left[ \left( \frac{(m+xn)}{\sqrt{y}}+r\cos \theta \right)^2 + (n\sqrt{y}+r\sin \theta)^2 \right]}}\\
&+ 4\pi^2 \alpha^2 \frac{\displaystyle \left(\sum_{m,n\in \Z} (n\sqrt{y}+r\sin \theta) e^{-\pi \alpha \left[ \left( \frac{(m+xn)}{\sqrt{y}}+r\cos \theta \right)^2 + (n\sqrt{y}+r\sin \theta)^2 \right]}\right)^2}{\displaystyle \sum_{m,n\in \Z} e^{-\pi \alpha \left[ \left( \frac{(m+xn)}{\sqrt{y}}+r\cos \theta \right)^2 + (n\sqrt{y}+r\sin \theta)^2 \right]}}.
\end{align*}

\begin{figure}[!h]
\centering
\includegraphics[width=6cm]{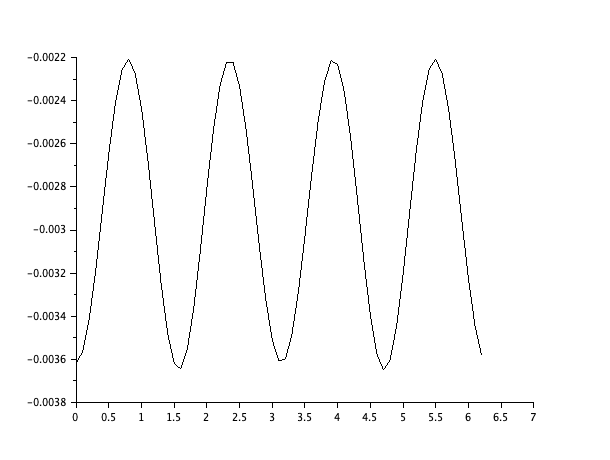} \quad \includegraphics[width=6cm]{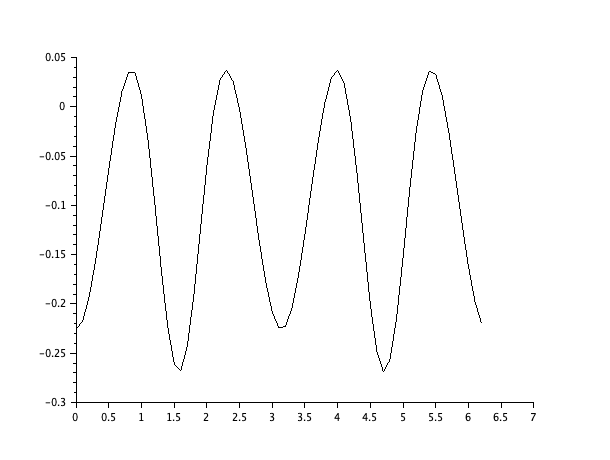}\\
\includegraphics[width=6cm]{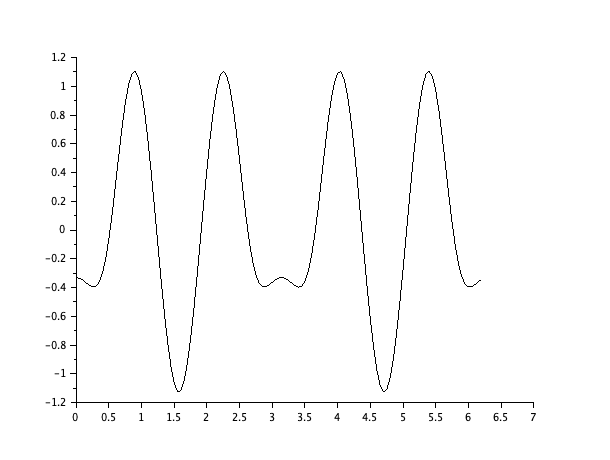} \quad \includegraphics[width=6cm]{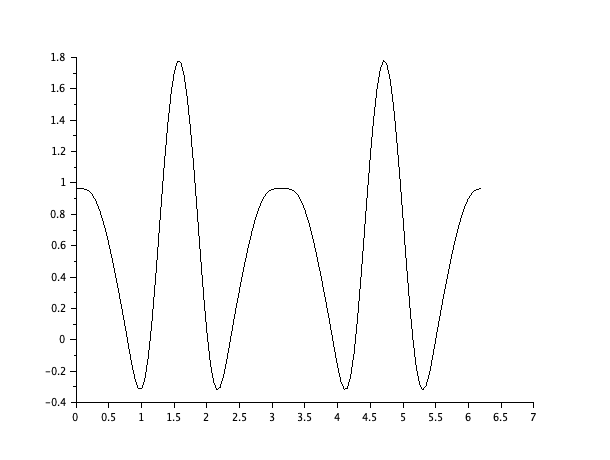}
\caption{Plots of $g_r$ (see \eqref{eq:gr}) on $[0,2\pi]$ for $\alpha=\frac{10}{\pi}$, $r\in\{0.1,0.3,0.5,0.6\}$ in reading direction.}
\label{Plotgr1}
\end{figure}


\begin{figure}[!h]
\centering
\includegraphics[width=6cm]{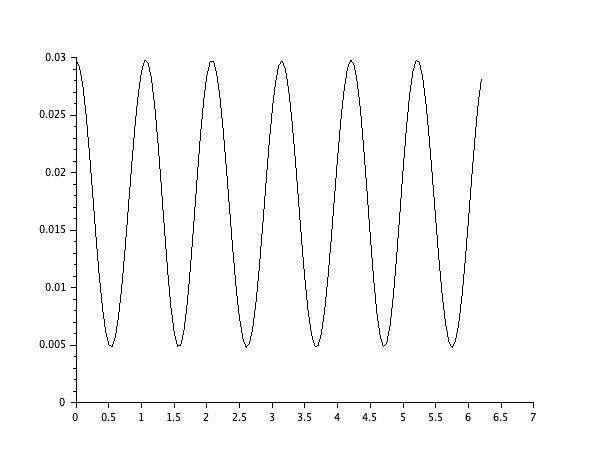} \quad \includegraphics[width=6cm]{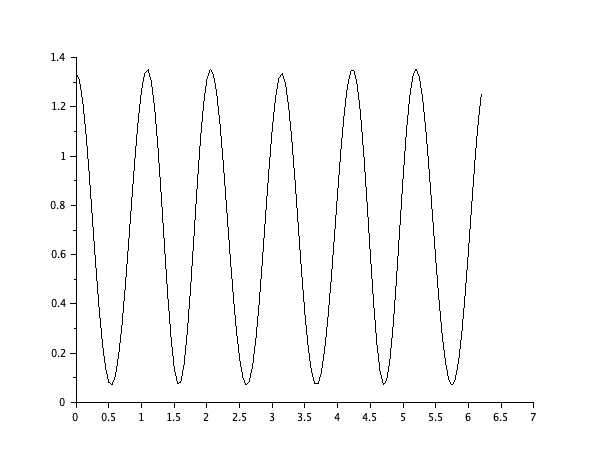}
\caption{Plot of $\theta\mapsto | \nabla_{L}\mathcal{Q}^\alpha_{L}(r\cos \theta, r\sin \theta) |^2$ on $[0,2\pi]$ for $L=\mathsf{A}_2$, $\alpha=\frac{10}{\pi}$ and fixed $r\in \{0.2,0.3\}$.}
\label{Gradient1}
\end{figure}



\subsection{Numerical results for $\mathcal{F}_\mu^\alpha$ in dimension 2}

We are considering the case $\alpha=\frac{10}{\pi}$  (see Figure \ref{Gaussian1graph}), in such a way that 
$$
\theta_{L+y}(\alpha)=\sum_{p\in L} e^{-10 |p+y|^2}.
$$
We are also assuming that $\mu=\sigma_R$ is the uniform measure on the disk $\Sigma=B_R$ for some $R>0$. For simplicity, we omit to renormalize the measure, in such a way that $\mu$ is almost never a probability measure. It does not matter in our study since we are considering optimizers at fixed $R$.\\

\begin{figure}[!h]
\centering
\includegraphics[width=6cm]{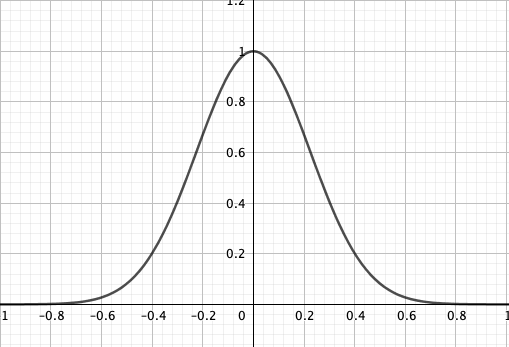} 
\caption{Plot of $r\mapsto f(r^2)=e^{-10r^2}$.}
\label{Gaussian1graph}
\end{figure}

We observe that for $R\in \{0.1,0.2,0.3,0.4,0.5\}$, the triangular lattice $\mathsf{A}_2$ is the unique maximizer of $\mathcal{F}_{\mu}^\alpha$ in $\mathcal{L}_2(1)$. In Figure \ref{Gaussian1}, we give one example of our numerical outputs as a contour plot for the particular case $R=0.5$. We have also checked that the same maximality result holds for $R=\frac{1}{2}\sqrt{\frac{2}{\sqrt{3}}}$ being the half side length of $\mathsf{A}_2$, i.e. the radius of the balls reaching the densest lattice packing of unit density in $\R^2$. If $R\geq 0.59$, the square lattice has a higher energy than the triangular one and $\mathcal{F}_{\mu}^\alpha$ does not have any maximizer.

\medskip

In order to illustrate Theorem \ref{thm:alpha0epsilon0}, we choose $\alpha=\frac{2}{\pi}$, $\Sigma=B_R$ $\mu=\sigma_R$, and we observe in Figure \ref{Gaussiana2R01}) that for $R=0.1$, $\mathsf{A}_2$ is the unique minimizer of $\mathcal{F}_{\mu}^\alpha$ in $\mathcal{L}_2(1)$ and the functional does not have any maximizer.

\begin{figure}[!h]
\centering
\includegraphics[width=10cm]{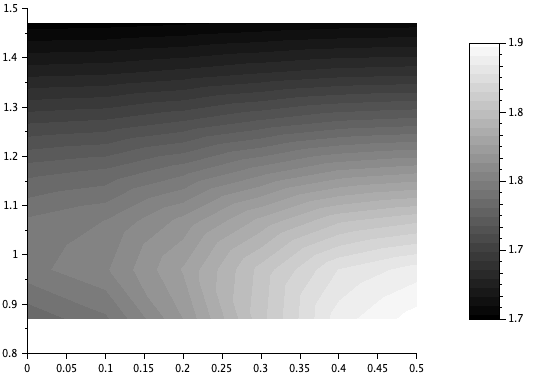} 
\caption{Case $\mu=\sigma_R$ uniform on $\Sigma=B_R$, $R=0.5$ and $\alpha=\frac{10}{\pi}$. The triangular lattice $\mathsf{A}_2$, corresponding to the point $(1/2,\sqrt{3}/2)$, is the unique maximizer of $\mathcal{F}_\mu^\alpha$ in $\mathcal{L}_2(1)$.}
\label{Gaussian1}
\end{figure}

\begin{figure}[!h]
\centering
\includegraphics[width=8cm]{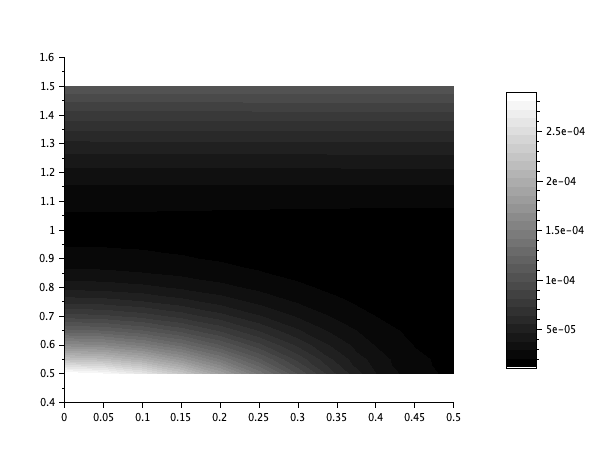} \quad \includegraphics[width=8cm]{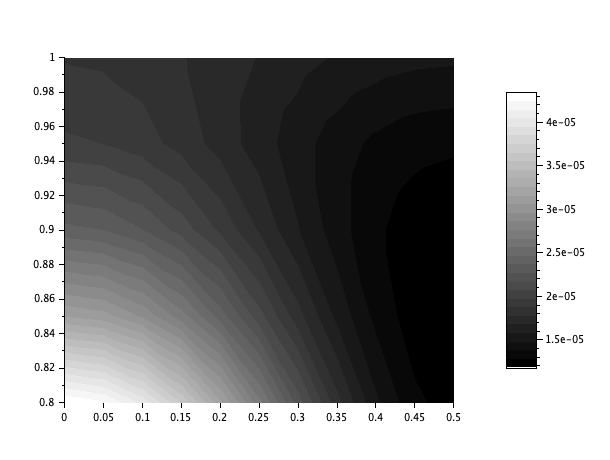} 
\caption{Illustration of Theorem \ref{thm:alpha0epsilon0}. Case $\mu=\sigma_R$ uniform on $\Sigma=B_R$, $R=0.1$ and $\alpha=\frac{2}{\pi}$. There is no maximizer (left). Furthermore, the triangular lattice $\mathsf{A}_2$, corresponding to the point $(1/2,\sqrt{3}/2)$, is the unique minimizer of $\mathcal{F}_\mu^\alpha$ in $\mathcal{L}_2(1)$.}
\label{Gaussiana2R01}
\end{figure}

If we compare the Fisher Information of the two-dimensional volume-stationary lattices, i.e. $\mathsf{A}_2$ and $\Z^2$ for fixed $\alpha=10/\pi$ and increasing $R$, we obtain the graph of Figure \ref{fig:a10Rup}. We observe that, for $R\in [0.1,0.57]$, $\mathcal{F}_{\sigma_R}^\alpha(\mathsf{A}_2)>\mathcal{F}_{\sigma_R}^\alpha(\Z^2)$ whereas the square lattice has a higher Fisher Information for $R\in [0.58,0.7]$. We also remark that the firing fields start to overlap when $R>\frac{1}{2}\sqrt{\frac{2}{\sqrt{3}}}\approx 0.5373$.

\begin{figure}[!h]
\centering
\includegraphics[width=10cm]{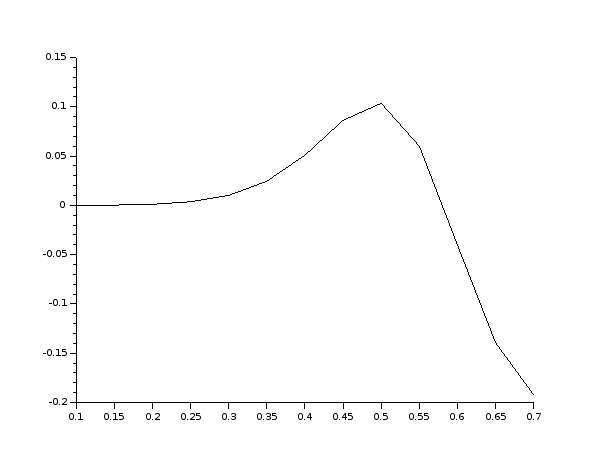} 
\caption{Plot of $R\mapsto \mathcal{F}_{\sigma_R}^\alpha(\mathsf{A}_2)-\mathcal{F}_{\sigma_R}^\alpha(\Z^2)$ where $\alpha=\frac{10}{\pi}$ and $\sigma_R$ is the uniform measure on $\Sigma=B_R$, for $R\in [0.1,0.7]$.}
\label{fig:a10Rup}
\end{figure}

\medskip

We have also numerically checked the case where the ratio between $2R$ and the lattice spacing of the triangular lattice is approximately $0.3$. Applied to the triangular lattice $\mathsf{A}_2$ of side length $\sqrt{2/\sqrt{3}}$, it means that $R=0.16$. For this value of $R$, we have found that the triangular lattice is always maximal for $\alpha\in (1.25,+\infty)$ (see Figure \ref{fig:expR}). Notice that even though a large zone of these figures seems to reach the maximum around the triangular lattice, we have checked that the minimum is really given by the latter.

\begin{figure}[!h]
\centering
\includegraphics[width=5cm]{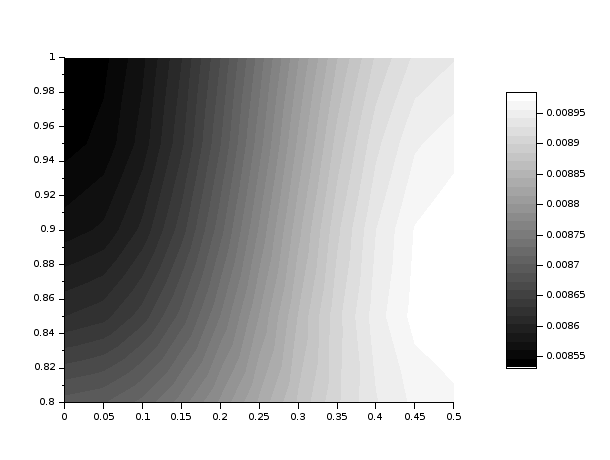} \quad \includegraphics[width=5cm]{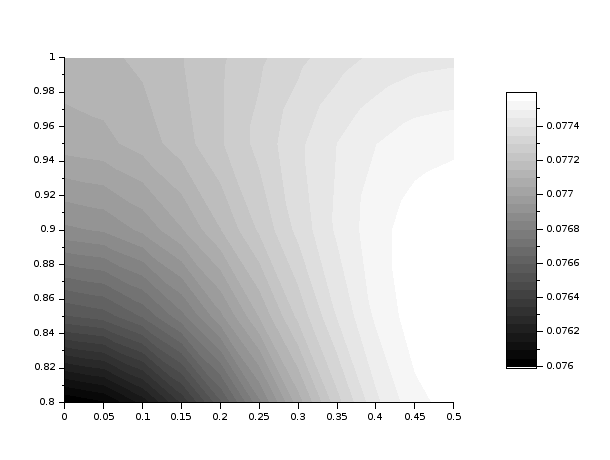}\quad \includegraphics[width=5cm]{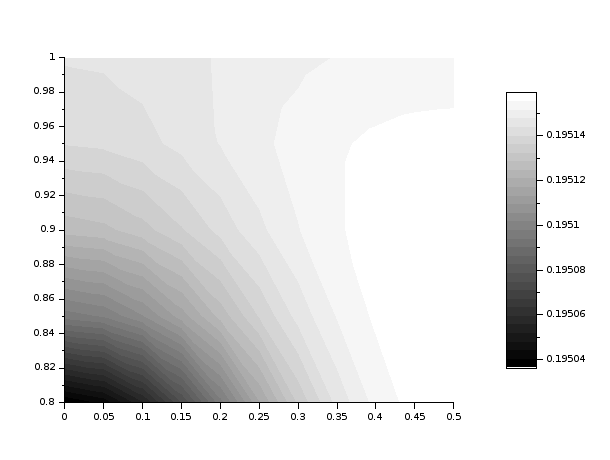}
\caption{Plot of $\mathcal{F}_{\sigma_R}^\alpha$ where $R=0.16$ and $\alpha\in \{1.3,3,5\}$. The triangular lattice appears to be a maximizer in all these cases.}
\label{fig:expR}
\end{figure}

\subsection{Numerical results for $\mathcal{F}_\mu^\alpha$ in dimension 3}

For our investigation in dimension $d=3$, we have again chosen $\alpha=\frac{10}{\pi}$. Since the numerical computations are very demanding, we are only giving the following type of results: some local optimality results and some comparisons of volume-stationary lattices, i.e. $\Z^3,\mathsf{D}_3$ and $\mathsf{D}_3^*$.

\medskip

As in the previous section on two-dimensional lattices, it is clear that $\mathsf{D}_3$, $\mathsf{D}_3^*$ and $\Z^3$ are almost never critical points for $L\mapsto \mathcal{Q}_L^\alpha(y)$ for a given $y$. It is only true for some points $y$ being the holes or deep holes of $L$ (center of cells, midpoints, etc.) as we have already explained for $\mathsf{A}_2$ (see Figure \ref{fig:exdiscrete}). It is therefore not sufficient to study $L\mapsto\mathcal{Q}_L^\alpha(y)$ in order to investigate the extrema of $\mathcal{F}_\mu^\alpha$.

\medskip

By comparing $\mathcal{Q}_L^\alpha(y)$ for $y\in \Sigma=B_R$, we observe that  there exists $R_0\approx 0.1$ such that 
$$
\mathcal{Q}_{\mathsf{D}_3}(y)>\mathcal{Q}_{\mathsf{D}_3^*}(y)>\mathcal{Q}_{\Z^3}(y),\quad \forall y\in B_{R_0},
$$
which directly means that, for $\mu=\sigma_{R_0}$, 
\begin{equation}\label{eq:3dcompar}
\mathcal{F}_\mu^\alpha(\mathsf{D}_3)>\mathcal{F}_\mu^\alpha(\mathsf{D}_3^*)>\mathcal{F}_\mu^\alpha(\Z^3).
\end{equation}
We have also checked on a set of discrete values of $R$ that \eqref{eq:3dcompar} still holds for $R\in [0.1,0.57]$. Notice that $0.57>2^{-5/6}=:R_{\mathsf{D}_3}$ which is the radius of spheres giving the densest packing of unit density in $\R^3$, on a FCC lattice.

\medskip

Furthermore, we have also numerically checked the local maximality of $\mathsf{D}_3$ for radii $R\in \{0.1,0.2,0.3,0.4,0.5,2^{-\frac{5}{6}}\}$ by computing the Hessian of $\mathcal{F}_\mu^\alpha$ at $L=\mathsf{D}_3$ which appears to be definite positive in all these cases.

\medskip
\medskip

\noindent \textbf{Acknowledgments:} I would like to thank the Austrian Science Fund (FWF) for its financial support through the project F65 as well as the DFG-FWF international joint project FR 4083/3-1/I 4354. I would also like to thank Martin Stemmler for helpful discussions about the size of the firing fields. I am also grateful to the anonymous referees for their suggestions.

{\small \bibliographystyle{plain}
\bibliography{Gridcells}}

\end{document}